\documentclass[10pt,b5paper]{article}
\usepackage[T1]{fontenc}
\usepackage[utf8]{inputenc}
\usepackage{amsmath,amssymb,amsthm}
\usepackage{geometry}
\newcommand{\E}{\mathbb E}
\newcommand{\Bin}{\operatorname{Bin}}
\newcommand{\Poi}{\operatorname{Poi}}
\newcommand{\NB}{\operatorname{NB}}
\newcommand{\Hyp}{\operatorname{Hyp}}
\newcommand{\Z}{\mathbb Z}
\newcommand{\R}{\mathbb R}

\theoremstyle{plain}
\newtheorem{theorem}{Theorem}[section]
\newtheorem{lemma}[theorem]{Lemma}
\newtheorem{proposition}[theorem]{Proposition}

\theoremstyle{remark}
\newtheorem{remark}[theorem]{Remark}
\theoremstyle{definition}

\title{The mean absolute deviation of the classical discrete distributions:\\
collapse identities, complete asymptotic expansions,\\
and enveloping series}
\author{
N.~Elezovi\'c\\
Department of Applied Mathematics,\\
Faculty of Electrical Engineering and Computing,\\
University of Zagreb, 10000 Zagreb, Croatia\\
\texttt{neven.elezovic@fer.hr}
}
\date{\today}

\begin{document}
\maketitle

\begin{abstract}
	For each of the four classical discrete laws --- binomial, Poisson, negative binomial and
	hypergeometric --- the mean absolute deviation about the mean collapses to a single point
	mass.  We give a common telescoping proof of these identities and interpret the resulting
	closed forms by size biasing.  We then derive complete asymptotic expansions for the Poisson ($\lambda\to\infty$),
	negative binomial ($r\to\infty$, $p$ fixed) and hypergeometric ($N\to\infty$, margins in fixed
	proportion) cases, extending the binomial expansion from the
	companion papers.  The coefficients are given in closed Bernoulli-polynomial form and carry
	the lattice displacement of the mean exactly.  At integer means the expansions reduce to
	sign-alternating odd series, and a single Binet-kernel argument shows that these series envelop the logarithm of the
	normalised mean absolute deviation: successive partial sums bracket it.
\end{abstract}



\section{Introduction}\label{sec:intro}

	Let $X$ be one of the classical discrete random variables, with mean $\mu$.  The
	\emph{mean absolute deviation}
	\[
		\E\,|X-\mu|
	\]
	is, for the binomial law, one of the oldest explicitly summed quantities in probability:
	De~Moivre's formula
	\begin{equation}\label{eq:demoivre}
		\E|X-Np|=2\nu q\,b(\nu;N,p),\qquad \nu=\lceil Np\rceil,\quad
		b(k;N,p)=\binom Nkp^kq^{N-k},
	\end{equation}
	reduces the whole expectation to a single point mass.  The history of \eqref{eq:demoivre} and
	of the family of identities it belongs to is told by Diaconis and
	Zabell~\cite{diaconis_zabell}; closed forms of the same kind for the other classical laws are
	classical as well (Ramasubban~\cite{ramasubban1958}, Crow~\cite{crow1958},
	Katti~\cite{katti1960}, Kamat~\cite{kamat1965}; they are recorded in Johnson, Kotz and
	Kemp~\cite{jkk} as eqs.\ (3.15), (4.19), (5.27) and the hypergeometric formula of
	Chapter~6).

	In the companion papers \cite{elezovic_mad,elezovic_mode} the binomial case was carried from
	the closed form to the \emph{complete asymptotic expansion}: since $Np$ is generally not a
	lattice point, the coefficients cannot be constants, and they turn out to be Bernoulli
	polynomials evaluated at the oscillating displacement $h_N=\lceil Np\rceil-Np$, given in
	closed form to all orders.  Two structural facts emerged there.  First, the elementary
	(non-Bernoulli) tail of the underlying Stirling expansion is removed \emph{exactly} by
	De~Moivre's prefactor $\nu$, and the mechanism is the size-bias identity
	$m\,b(m;N,p)=Np\,b(m-1;N-1,p)$: the prefactor \emph{is} the size-bias factor
	\cite{elezovic_mode}.  Second, at an integer mean the expansion collapses to a
	sign-alternating series in odd powers of $N^{-1}$, which is \emph{enveloping}: successive
	truncations bracket the value, by a sign-definite Binet-kernel representation of the combined
	Stirling remainder \cite{elezovic_mad}.

	The present paper extends these mechanisms to the Poisson, negative binomial and
	hypergeometric laws.  The first result is a common telescoping identity
	\[
		(k-\mu)P(k)=g(k)P(k)-g(k+1)P(k+1)
	\]
	with $g$ of the degree dictated by the Katz--Ord ratio.  It gives
	$\E|X-\mu|=2g(\nu)P(\nu)$, $\nu=\lceil\mu\rceil$, and through the standard size-bias
	identities the four closed forms assume the uniform shape
	\[
		\E|X-\mu|=2\,\sigma^{2}\,P^{\downarrow}\{\nu-1\},
	\]
	where $P^{\downarrow}$ is the corresponding size-biased law.

	The second result is the complete asymptotic expansion for the three non-binomial laws.  The
	coefficients are Bernoulli polynomials at the lattice defect of the mean, with law-specific
	weights.  In the Poisson case,
	\[
		\E|X-\lambda|\;\sim\;\sqrt{\frac{2\lambda}{\pi}}\;
		\exp\Bigl(-\sum_{n\ge1}\frac{B_{n+1}(\{\lambda\})}{n(n+1)}\,\frac1{\lambda^{n}}\Bigr);
	\]
	there is one Bernoulli polynomial per order and no elementary cancellation.  In the negative
	binomial and hypergeometric cases the elementary tail is removed exactly by the collapse
	prefactor $g(\nu)$, respectively by one and by two size-bias factors.

	The third result concerns integer means.  In each law the expansion reduces to a
	sign-alternating series in odd powers, and the series is enveloping.  The proof is reduced to
	one Binet-kernel statement, Proposition~\ref{prop:skeleton}: the logarithm of the normalised
	mean absolute deviation is $\int_0^\infty\varphi(t)\Delta(t)\,dt$, where the law-specific
	combination $\Delta$ of exponentials is strictly negative.

	Positioning with respect to the literature is discussed in Section~\ref{sec:conclusion}; in
	brief, the closed forms of Section~\ref{sec:collapse} are classical case by case
	(Ramasubban~\cite{ramasubban1958}, Crow~\cite{crow1958}, Katti~\cite{katti1960},
	Kamat~\cite{kamat1965}), and what we claim there is the uniform statement and the
	telescoping that produces it, while the complete expansions and the enveloping theorems
	appear to be new beyond the binomial case.

\section{Preliminaries}\label{sec:prelim}

	$B_k$ and $B_k(t)$ denote the Bernoulli numbers and polynomials in the standard
	normalisation
	\[
		\frac{z\,e^{tz}}{e^{z}-1}=\sum_{k\ge0}B_k(t)\,\frac{z^k}{k!},
		\qquad B_k=B_k(0),\qquad B_1=-\tfrac12 ,
	\]
	with the two structural identities
	\begin{equation}\label{eq:appell}
		B_m(t+1)-B_m(t)=m\,t^{m-1},
		\qquad
		B_m(1-t)=(-1)^{m}B_m(t) .
	\end{equation}
	In particular $B_m(1)=B_m$ for $m\ge2$.

	Asymptotic statements are complete and uniform in the sense of
	\cite[\S1]{elezovic_mad}: $f\sim g\sum c_nx^{-n}$ means that for every fixed $M$,
	$f=g(\sum_{n\le M}c_nx^{-n}+O(x^{-M-1}))$ with the implied constant uniform over the
	parameter compacts specified in the statement; when the answer is displayed as the
	exponential of a series, the logarithmic series is truncated at a fixed order and the
	exponential re-expanded, the two forms carrying the same information through the standard
	recursion: if $\exp(\sum_{k\ge1}a_kx^{-k})=\sum_{m\ge0}d_mx^{-m}$ then
	\begin{equation}\label{eq:bell}
		d_0=1,\qquad d_m=\frac1m\sum_{k=1}^{m}k\,a_k\,d_{m-k}.
	\end{equation}

	We shall use two results of \cite{elezovic_mad} as black boxes.  The first is the shifted
	(Appell) form of Stirling's series,
	\begin{equation}\label{eq:stirling-shift}
	\begin{aligned}
		\log\Gamma(x+t)
		&\sim\Bigl(x+t-\tfrac12\Bigr)\log x-x+\tfrac12\log2\pi\\
		&\quad+\sum_{n\ge1}\frac{(-1)^{n+1}B_{n+1}(t)}{n(n+1)}\,\frac1{x^{n}},
		\qquad x\to\infty,
	\end{aligned}
	\end{equation}
	with, for every $M$, a remainder $O(x^{-M-1})$ uniform for the shift $t$ in compact sets
	\cite[\S2]{elezovic_mad}.  The second is the expansion of a gamma quotient with unequal
	scalings, which we restate for convenience.

\begin{lemma}[{\cite[Lemma 2.1]{elezovic_mad}}]\label{lem:unequal}
	Let $\lambda_j,\mu_k>0$ and $u_j,v_k\in\R$, where $1\le j\le\rho$ and
	$1\le k\le\sigma$.  Put
	\[
		F(x)=\frac{\prod_j\Gamma(\lambda_jx+u_j)}
			{\prod_k\Gamma(\mu_kx+v_k)} .
	\]
	If $\Lambda:=\sum_j\lambda_j-\sum_k\mu_k=0$, then, as $x\to\infty$,
	\begin{equation}\label{eq:logF}
		\log F(x)\sim \Theta\,x+U\log x+K
		+\sum_{n\ge1}\frac{(-1)^{n+1}}{n(n+1)}\,S_{n+1}\,\frac{1}{x^{n}},
	\end{equation}
	where
	\[
		\Theta=\sum_j\lambda_j\log\lambda_j-\sum_k\mu_k\log\mu_k,\qquad
		U=\sum_j\Bigl(u_j-\tfrac12\Bigr)-\sum_k\Bigl(v_k-\tfrac12\Bigr),
	\]
	\[
		K=\sum_j\Bigl(u_j-\tfrac12\Bigr)\log\lambda_j
		 -\sum_k\Bigl(v_k-\tfrac12\Bigr)\log\mu_k+\frac{\rho-\sigma}{2}\log2\pi,
	\]
	\begin{equation}\label{eq:Sn}
		S_{n+1}=\sum_{j}\lambda_j^{-n}B_{n+1}(u_j)-\sum_{k}\mu_k^{-n}B_{n+1}(v_k),
	\end{equation}
	and, for every $M$, the remainder after $M$ series terms is $O(x^{-M-1})$, uniformly for the
	shifts in compact sets and the scalings in compact subsets of $(0,\infty)$.
\end{lemma}

	Finally, one elementary lemma on logarithmic factors, used for the tail cancellations; it is
	the argument of \cite[\S2]{elezovic_mode} in the form we need.

\begin{lemma}\label{lem:logfactor}
	Let $\alpha$ range over a compact subset of $(0,\infty)$ and $h$ over a compact subset of
	$[0,\infty)$.  Then, for $x$ so large that $h/(\alpha x)\le\tfrac12$ on the compacts,
	\begin{equation}\label{eq:logfactor}
		\sum_{k=1}^{M}\frac{(-1)^{k-1}}{k}\Bigl(\frac{h}{\alpha x}\Bigr)^{k}
		=\log\Bigl(1+\frac{h}{\alpha x}\Bigr)+O\bigl(x^{-M-1}\bigr),
	\end{equation}
	for every $M\ge1$, with an implied constant depending only on $M$ and the compacts; and the
	factor $1+h/(\alpha x)$ is bounded above and below by positive constants there.
\end{lemma}

\begin{proof}
	The series $\sum_{k\ge1}(-1)^{k-1}z^k/k=\log(1+z)$ converges for $|z|\le\tfrac12$ with tail
	bounded by $2|z|^{M+1}/(M+1)$; put $z=h/(\alpha x)=O(x^{-1})$.
\end{proof}

\section{The collapse}\label{sec:collapse}

	Throughout this section $X$ is one of
	\[
		\Bin(N,p),\qquad \Poi(\lambda),\qquad \NB(r,p),\qquad \Hyp(N,K,n),
	\]
	with the negative binomial in the \emph{failures} parametrisation,
	$P(k)=\binom{k+r-1}{k}p^{r}q^{k}$, $q=1-p$, mean $\mu=rq/p$; and the hypergeometric
	$P(k)=\binom Kk\binom{N-K}{n-k}\big/\binom Nn$, mean $\mu=nK/N$.  In each case write
	$P(k)=\Pr\{X=k\}$ and let $\mu$ denote the mean.  We assume the non-degeneracy conditions $0<p<1$, $\lambda>0$, $r>0$ and, for the hypergeometric, $1\le K\le N-1$ and $1\le n\le N-1$; these keep the parameters and the size-bias-shifted laws of Theorem~\ref{thm:universal} well defined.  In the excluded boundary cases the mean absolute deviation vanishes and the identities hold trivially.

\begin{theorem}[One telescoping for the family]\label{thm:collapse}
	Define
	\begin{equation}\label{eq:g-table}
		g(k)=
		\begin{cases}
			kq, & X\sim\Bin(N,p),\\[2pt]
			k, & X\sim\Poi(\lambda),\\[2pt]
			k/p, & X\sim\NB(r,p),\\[4pt]
			\dfrac{k\,(k+N-K-n)}{N}, & X\sim\Hyp(N,K,n).
		\end{cases}
	\end{equation}
	Then in each case
	\begin{equation}\label{eq:tele}
		(k-\mu)\,P(k)=g(k)P(k)-g(k+1)P(k+1)\qquad(k\ge0),
	\end{equation}
	and consequently, for \emph{every} integer $\nu\ge0$,
	\begin{equation}\label{eq:tele-sum}
		\sum_{k\ge\nu}(k-\mu)\,P(k)=g(\nu)\,P(\nu) .
	\end{equation}
	In particular, with $\nu=\lceil\mu\rceil$,
	\begin{equation}\label{eq:collapse}
		\E\,|X-\mu|=2\,g(\nu)\,P(\nu).
	\end{equation}
	When $\mu\in\Z$, the two choices $\nu=\mu$ and $\nu=\mu+1$ in \eqref{eq:collapse} give the
	same value.
\end{theorem}

\begin{proof}
	Write $\rho_k:=P(k+1)/P(k)$ wherever $P(k)>0$.  Dividing \eqref{eq:tele} by $P(k)$, the
	claim is the identity
	\begin{equation}\label{eq:tele-crit}
		g(k)-g(k+1)\,\rho_k=k-\mu .
	\end{equation}
	We verify \eqref{eq:tele-crit} case by case; outside the support both sides of
	\eqref{eq:tele} vanish, so \eqref{eq:tele-crit} on the support suffices.

	\emph{Binomial}: $\rho_k=\dfrac{(N-k)p}{(k+1)q}$, so
	\[
		kq-(k+1)q\cdot\frac{(N-k)p}{(k+1)q}=kq-(N-k)p=k(p+q)-Np=k-\mu .
	\]

	\emph{Poisson}: $\rho_k=\dfrac{\lambda}{k+1}$, so
	$k-(k+1)\cdot\dfrac{\lambda}{k+1}=k-\lambda$.

	\emph{Negative binomial}: $\rho_k=\dfrac{q(k+r)}{k+1}$, so
	\[
		\frac kp-\frac{k+1}{p}\cdot\frac{q(k+r)}{k+1}
		=\frac{k-q(k+r)}{p}=\frac{k(1-q)-qr}{p}=k-\frac{qr}{p}=k-\mu .
	\]

	\emph{Hypergeometric}: write $L:=N-K-n$, so $g(k)=k(k+L)/N$ and
	$\rho_k=\dfrac{(K-k)(n-k)}{(k+1)(L+k+1)}$.  Then
	\begin{align*}
		g(k)-g(k+1)\rho_k
		&=\frac{k(k+L)}{N}
		 -\frac{(k+1)(k+1+L)}{N}\cdot\frac{(K-k)(n-k)}{(k+1)(L+k+1)}\\
		&=\frac{k(k+L)-(K-k)(n-k)}{N}.
	\end{align*}
	Expanding the numerator, the quadratic terms cancel:
	\[
		k^2+kL-\bigl(Kn-(K+n)k+k^2\bigr)=k\,(L+K+n)-Kn=kN-Kn,
	\]
	so $g(k)-g(k+1)\rho_k=k-nK/N=k-\mu$.

	For \eqref{eq:tele-sum}, sum \eqref{eq:tele} over $\nu\le k\le L$; the sum telescopes to
	$g(\nu)P(\nu)-g(L+1)P(L+1)$.  For the binomial and the hypergeometric the support is finite
	and $g(L+1)P(L+1)=0$ once $L+1$ leaves it.  For the Poisson and the negative binomial,
	$g(L+1)$ grows at most quadratically while $P(L+1)$ decays superexponentially
	(resp.\ geometrically, since $\rho_k\to q<1$), so $g(L+1)P(L+1)\to0$ and the sum converges
	absolutely.

	For \eqref{eq:collapse}: since $\E(X-\mu)=0$ we have
	$\E|X-\mu|=2\,\E(X-\mu)^{+}=2\sum_{k>\mu}(k-\mu)P(k)$.  If $\mu\notin\Z$ the summation range
	$k>\mu$ is exactly $k\ge\lceil\mu\rceil$; if $\mu\in\Z$ the term $k=\mu$ vanishes, so the sum
	over $k\ge\mu$ and over $k\ge\mu+1$ agree, which also proves the final claim by
	\eqref{eq:tele-sum}.
\end{proof}

\begin{remark}\label{rem:ord-degree}
	The degree of $g$ matches the degree of the numerator and denominator of the Katz--Ord ratio
	$\rho_k$: linear for the three laws with linear ratio, quadratic for the hypergeometric.
	This is forced by \eqref{eq:tele-crit}: if $\rho_k=A(k)/B(k)$ with $\deg A=\deg B=\delta$ and
	equal leading coefficients (as here), then $g$ of degree $\delta$ is exactly what makes the
	leading terms of $g(k)B(k)-g(k+1)A(k)$ cancel down to the linear right-hand side.
\end{remark}

	The closed forms \eqref{eq:collapse} are classical case by case --- in particular the
	hypergeometric case of \eqref{eq:collapse}, with the quadratic $g$, is exactly formula (3.3)
	of Ramasubban~\cite[p.~554]{ramasubban1958}; see Section~\ref{sec:conclusion} for the full
	attributions.  The telescoping \eqref{eq:tele} is not new either: summing it is the cumulative
	(``collapse'') identity of the discrete Pearson\,/\,Ord family,
	$\sum_{j\le x}(\mu-j)P(j)=q(x)P(x)$ where $q$ is the variance function of the family --- of degree at most two, quadratic for the
	hypergeometric and linear or constant for the other three --- and our carrier $g$ is its
	equivalent form, $q(x)=g(x+1)\rho_x$.  This is the discrete $w$-function of Cacoullos and
	Papathanasiou~\cite{cacoullos_pap1989}; it is tied directly to the mean absolute deviation by
	Korwar~\cite{korwar1991}, and its difference form $\Delta[q(x)P(x)]=(\mu-(x+1))P(x+1)$,
	equivalent to \eqref{eq:tele}, is that of Afendras, Papadatos and
	Papathanasiou~\cite[eq.~(2.1)]{afendras_pp2011}; the Ord degree of $g$
	(Remark~\ref{rem:ord-degree}) is intrinsic to it.  What we take from \eqref{eq:tele} is its use
	for the mean absolute deviation: the single passage to the collapse \eqref{eq:collapse} and,
	through the size-bias identities, to the uniform form of Theorem~\ref{thm:universal}.

\begin{theorem}[The universal form]\label{thm:universal}
	In each of the four cases, with $\nu=\lceil\mu\rceil$ and $\sigma^{2}$ the variance of $X$,
	\begin{equation}\label{eq:universal}
		\E\,|X-\mu|\;=\;2\,\sigma^{2}\;P^{\downarrow}\{\nu-1\},
	\end{equation}
	where $P^{\downarrow}$ denotes the mass function of
	\[
		\Bin(N-1,p),\qquad
		\Poi(\lambda),\qquad
		\NB(r+1,p),\qquad
		\Hyp(N-2,\,K-1,\,n-1),
	\]
	respectively.
\end{theorem}

\begin{proof}
	In each case we transform $g(\nu)P(\nu)$ by absorption identities.

	\emph{Binomial.}  $\nu\binom N\nu=N\binom{N-1}{\nu-1}$ gives
	$\nu\,b(\nu;N,p)=Np\,b(\nu-1;N-1,p)$, so
	\[
		g(\nu)P(\nu)=\nu q\,b(\nu;N,p)
		=Npq\,b(\nu-1;N-1,p)=\sigma^{2}P^{\downarrow}\{\nu-1\}.
	\]

	\emph{Poisson.}  $\nu P(\nu)=\nu e^{-\lambda}\lambda^{\nu}/\nu!
	=\lambda\,e^{-\lambda}\lambda^{\nu-1}/(\nu-1)!=\lambda P(\nu-1)$, and $\sigma^{2}=\lambda$.

	\emph{Negative binomial.}  From
	\[
		k\binom{k+r-1}{k}=r\binom{k+r-1}{k-1}
	\]
	and $\binom{k+r-1}{k-1}=\binom{(k-1)+(r+1)-1}{k-1}$,
	\[
	\begin{aligned}
		k\,P(k;r,p)
		&=r\binom{k+r-1}{k-1}p^{r}q^{k}\\
		&=\frac{rq}{p}\,\binom{(k-1)+(r+1)-1}{k-1}p^{r+1}q^{k-1}\\
		&=\frac{rq}{p}\,P(k-1;r+1,p).
	\end{aligned}
	\]
	Hence $g(\nu)P(\nu)=(\nu/p)P(\nu)=(rq/p^{2})P(\nu-1;r+1,p)$, and $\sigma^{2}=rq/p^{2}$.

	\emph{Hypergeometric.}  Two absorption steps are needed, one per factor of the quadratic
	$g(\nu)=\nu(\nu+L)/N$, $L=N-K-n$.  First, the sample size bias
	\[
		k\,P(k;N,K,n)=\frac{nK}{N}\,P(k-1;\,N-1,K-1,n-1),
	\]
	which follows from $k\binom Kk=K\binom{K-1}{k-1}$,
	$\binom{N-K}{n-k}=\binom{(N-1)-(K-1)}{(n-1)-(k-1)}$ and
	$\binom Nn=\tfrac Nn\binom{N-1}{n-1}$.  Applying it at $k=\nu$,
	\[
	\begin{aligned}
		g(\nu)P(\nu)
		&=\frac{\nu+L}{N}\cdot\nu P(\nu)
		=\frac{\mu\,(\nu+L)}{N}\;P'(\nu-1),\\
		P'&:=\Hyp(N-1,K-1,n-1).
	\end{aligned}
	\]
	Second, the complement size bias for $P'$.  Write $N'=N-1$, $K'=K-1$, $n'=n-1$ and
	$j=\nu-1$; then
	\[
		N'-K'-n'+j=N-K-n+\nu=\nu+L ,
	\]
	so the factor $\nu+L$ is exactly the count of the fourth cell of $P'$ at $j$.  From
	$(m-i)\binom mi=m\binom{m-1}i$ with $m=N'-K'$, $i=n'-j$, and
	$\binom{N'-1}{n'}\big/\binom{N'}{n'}=(N'-n')/N'$,
	\[
	\begin{aligned}
		(N'-K'-n'+j)\,P'(j)
		&=(N'-K')\,\frac{\binom{K'}{j}\binom{N'-K'-1}{n'-j}}{\binom{N'}{n'}}\\
		&=\frac{(N'-K')(N'-n')}{N'}\;P''(j),
	\end{aligned}
	\]
	where $P'':=\Hyp(N'-1,K',n')$.
	With $N'-K'=N-K$, $N'-n'=N-n$, $N'=N-1$ and $P''=\Hyp(N-2,K-1,n-1)$, combining the two
	steps gives
	\[
	\begin{aligned}
		g(\nu)P(\nu)
		&=\frac{\mu}{N}\cdot\frac{(N-K)(N-n)}{N-1}\;P''(\nu-1)\\
		&=\frac{nK(N-K)(N-n)}{N^{2}(N-1)}\;P''(\nu-1)
		=\sigma^{2}\,P''(\nu-1).
	\end{aligned}
	\]
\end{proof}

\begin{remark}\label{rem:universal-reading}
	Identity \eqref{eq:universal} is the family-wide version of the size-bias reading of
	De~Moivre's formula in \cite{elezovic_mode}: the mean absolute deviation is twice the
	variance times a single mass of the size-bias-shifted law, evaluated one step below the
	ceiling of the mean.  In the language of the discrete $w$-function this is Korwar's
	mean-deviation identity~\cite{korwar1991}, $\tfrac12\E|X-\mu|=q(\nu-1)P(\nu-1)$, read through the
	size-bias transform so that the parameter shift becomes explicit for each member.  The Poisson
	needs no shift because it is its own size bias; the
	hypergeometric needs two, one from the sample and one from the complement --- and, as we
	shall see in Section~\ref{sec:hyp}, this doubling is precisely mirrored in the asymptotics,
	where the quadratic $g$ supplies \emph{two} elementary log-factors, one for each diagonal
	cell of the $2\times2$ table.  For the asymptotics it is \eqref{eq:collapse} that we use;
	\eqref{eq:universal} identifies the normalisation in which the expansions come out pure.
\end{remark}

\begin{remark}[The exact repair of an approximate identity]\label{rem:katz-approx}
	A unified mean-deviation formula for the Katz family (binomial, Poisson, negative binomial)
	is in fact on record: summing the basic difference identity of the family gives
	\[
		\E|X-\mu|\;\approx\;2\,\sigma^{2}\,P\{X=\lfloor\mu\rfloor\},
	\]
	\emph{with an explicit error term}, the formula being exact only when $\mu\in\Z$ or in the
	Poisson case; see \cite[eq.\ (2.53)]{jkk}, Bardwell~\cite{bardwell1960} and
	Kamat~\cite{kamat1965}.  Theorem~\ref{thm:universal} is the exact repair of this statement,
	across the full family including the hypergeometric: replacing the law by its
	size-bias-shifted version and evaluating at $\nu-1$ absorbs the error entirely, at every
	$\mu$.
\end{remark}

\begin{remark}[The continuous ancestor]\label{rem:bortkiewicz}
	The shape ``mean absolute deviation $=$ twice the variance times a density value at the
	mean'' has a continuous ancestor: von Bortkiewicz (1923) observed that the ratio
	$\E|X-\mu|\big/\bigl(2\sigma^{2}f(\mu)\bigr)$ equals $1$ \emph{exactly} for the normal,
	gamma and exponential laws, and Pearson-family generalisations followed; see Diaconis and
	Zabell \cite[\S3.3]{diaconis_zabell}, who also record that Ramasubban rediscovered the
	phenomenon for the Poisson.  In the continuous cases exactness holds for some families and
	fails by an explicit ratio for others (e.g.\ $(\alpha+\beta+1)/(\alpha+\beta)$ for the
	beta).  Identity \eqref{eq:universal} is the lattice-exact discrete counterpart: exactness
	holds for \emph{all four} classical laws, the correction being absorbed not by a ratio but
	by the size-bias parameter shift and the evaluation at $\nu-1$.
\end{remark}

\section{The Poisson law}\label{sec:poisson}

	Let $X\sim\Poi(\lambda)$ and write
	\[
		\theta:=\{\lambda\}\in[0,1),\qquad m:=\lfloor\lambda\rfloor,
	\]
	so that $\nu=\lceil\lambda\rceil$ and $\nu-1=m$ for $\lambda\notin\Z$, while for
	$\lambda\in\Z$ both readings of Theorem~\ref{thm:collapse} give, by
	$P(\lambda)=P(\lambda-1)$,
	\begin{equation}\label{eq:poi-closed}
		\E|X-\lambda|=2\lambda\,P\{X=m\}=2\,e^{-\lambda}\frac{\lambda^{m+1}}{m!}
		\qquad\text{in all cases.}
	\end{equation}

\subsection{The complete expansion}

\begin{theorem}\label{thm:poi-expansion}
	As $\lambda\to\infty$ through the reals,
	\begin{equation}\label{eq:poi-mad}
		\E|X-\lambda|\;\sim\;\sqrt{\frac{2\lambda}{\pi}}\;
		\exp\Bigl(\sum_{n\ge1}\frac{a_n(\theta)}{\lambda^{n}}\Bigr),
		\qquad
		a_n(\theta)=-\frac{B_{n+1}(\theta)}{n(n+1)} ,
	\end{equation}
	in the following sense: for every $M$,
	\[
		\E|X-\lambda|=\sqrt{\frac{2\lambda}{\pi}}
		\Bigl(\sum_{n=0}^{M}\frac{d_n(\theta)}{\lambda^{n}}+O(\lambda^{-M-1})\Bigr),
	\]
	where the $d_n$ are obtained from the $a_n$ by \eqref{eq:bell}, and the implied constant
	depends only on $M$ --- the estimate is uniform in $\theta\in[0,1)$, hence in $\lambda$ over any
	$[\lambda_0,\infty)$ with $\lambda_0>0$.  The first multiplicative coefficients are
	\[
		d_1=-\tfrac12B_2(\theta),\qquad
		d_2=-\tfrac16B_3(\theta)+\tfrac18B_2(\theta)^{2}.
	\]
\end{theorem}

\begin{proof}
	By \eqref{eq:poi-closed},
	\[
		\log\E|X-\lambda|=\log(2\lambda)-\lambda+m\log\lambda-\log\Gamma(m+1).
	\]
	Since $m+1=\lambda+(1-\theta)$ with $1-\theta\in(0,1]$ ranging over a compact set, the
	shifted Stirling series \eqref{eq:stirling-shift} applies with $x=\lambda$, $t=1-\theta$:
	\[
	\begin{aligned}
		\log\Gamma(m+1)
		&=\Bigl(\lambda+\tfrac12-\theta\Bigr)\log\lambda-\lambda+\tfrac12\log2\pi\\
		&\quad+\sum_{n=1}^{M}\frac{(-1)^{n+1}B_{n+1}(1-\theta)}
			{n(n+1)}\,\frac1{\lambda^{n}}
		+O(\lambda^{-M-1}),
	\end{aligned}
	\]
	uniformly in $\theta$.  Substituting, with $m\log\lambda=(\lambda-\theta)\log\lambda$, the
	terms $\pm\lambda$ and $\pm\lambda\log\lambda$, $\mp\theta\log\lambda$ cancel, leaving
	\[
	\begin{aligned}
		\log\E|X-\lambda|=\log(2\lambda)-\tfrac12\log\lambda-\tfrac12\log2\pi
		&-\sum_{n=1}^{M}\frac{(-1)^{n+1}B_{n+1}(1-\theta)}{n(n+1)\lambda^{n}}\\
		&+O(\lambda^{-M-1}).
	\end{aligned}
	\]
	The elementary part is $\log\sqrt{2\lambda/\pi}$.  By the reflection identity
	\eqref{eq:appell},
	\[
		(-1)^{n+1}B_{n+1}(1-\theta)=B_{n+1}(\theta),
	\]
	and \eqref{eq:poi-mad} follows.
	Exponentiating a truncation, as in \S\ref{sec:prelim}, preserves the
	relative $O(\lambda^{-M-1})$; the coefficients $d_1,d_2$ follow from \eqref{eq:bell}.
\end{proof}

\begin{remark}\label{rem:poi-clean}
	Two features distinguish the Poisson within the family.  First, \emph{no cancellation is
	needed}: the collapse prefactor in \eqref{eq:poi-closed} is the constant $2\lambda$, not the
	summation index, because the Poisson is its own size bias; correspondingly there is a single
	large gamma argument, and the naive expansion is already pure.  (For the binomial the
	prefactor is $\nu$ and the elementary tail it removes is the price of \emph{two} large gamma
	arguments \cite{elezovic_mad,elezovic_mode}; the negative binomial and hypergeometric below
	behave likewise.)  Second, the coefficient of $\lambda^{-n}$ is a \emph{single} Bernoulli
	polynomial with weight one --- compare the binomial's weight $p^{-n}+(-1)^{n+1}q^{-n}$.

	The parameter is continuous, and this simplifies the averaging of the oscillation: the mean
	of $B_j(\{\lambda\})$ over any period $\lambda\in[\Lambda,\Lambda+1]$ is
	$\int_0^1B_j=0$ exactly, so the continuous Ces\`aro mean of every coefficient is its smooth
	part, with none of the rational/irrational dichotomy of the lattice-parameter case
	\cite[\S6]{elezovic_mad}.  The arithmetic returns only if $\lambda$ is restricted to an
	arithmetic sequence.
\end{remark}

\subsection{The enveloping theorem}

	We first isolate the kernel argument; it is the argument of the enveloping theorem of
	\cite[\S5]{elezovic_mad}, stated here once so that all three laws can use it.  That a series is
	\emph{enveloping} in the sense of P\'olya and Szeg\H o~\cite{polya_szego} --- successive
	truncations bracket the value, the error having the sign of, and being bounded by, the first
	omitted term --- is the sense in which Brent~\cite{brent2020} showed the asymptotic series of
	$\log\binom{2n}{n}$ and of Binet's function to be enveloping; the central binomial coefficient
	is exactly the binomial mean-deviation prefactor.  What is new here is the enveloping of the
	mean absolute deviation of the Poisson, negative binomial and hypergeometric laws, all through
	the single kernel of Proposition~\ref{prop:skeleton}.

	Let
	\begin{equation}\label{eq:phi}
		\varphi(t):=\frac1t\Bigl(\frac1{e^{t}-1}-\frac1t+\frac12\Bigr),
		\qquad t>0,
	\end{equation}
	the kernel of Binet's first formula \cite[\S12.31]{ww}:
	\begin{equation}\label{eq:binet}
	\begin{aligned}
		\log\Gamma(x)
		&=\Bigl(x-\tfrac12\Bigr)\log x-x+\tfrac12\log2\pi+J(x),\\
		J(x)&=\int_0^\infty\varphi(t)\,e^{-xt}\,dt\qquad(x>0).
	\end{aligned}
	\end{equation}

	Here $J$ is Binet's function; we use the letter $J$ rather than the traditional $\mu$ so that
	$\mu$ can denote the mean throughout.

\begin{proposition}[Enveloping skeleton]\label{prop:skeleton}
	Let $\Delta:(0,\infty)\to\R$ be continuous with $\Delta(t)<0$ for all $t>0$ and
	$|\Delta(t)|\le Ce^{-ct}$ for some $c>0$.  Put
	\[
		L:=\int_0^\infty\varphi(t)\,\Delta(t)\,dt,
		\qquad
		c_n:=\frac{B_{2n}}{(2n)!}\int_0^\infty t^{2n-2}\,\Delta(t)\,dt\qquad(n\ge1).
	\]
	All the integrals converge absolutely, and for every $M\ge0$
	\begin{equation}\label{eq:skeleton}
		(-1)^{M+1}\Bigl[L-\sum_{n=1}^{M}c_n\Bigr]\;>\;0 .
	\end{equation}
	Consequently $L$ lies strictly between any two consecutive partial sums of $\sum_nc_n$, and
	$\operatorname{sign}c_n=(-1)^{n}$.
\end{proposition}

\begin{proof}
	By the Mittag--Leffler expansion of the hyperbolic cotangent \cite[\S7.4]{ww},
	\[
		\frac1{e^{t}-1}-\frac1t+\frac12=\sum_{k\ge1}\frac{2t}{t^{2}+4\pi^{2}k^{2}},
		\qquad\text{so}\qquad
		\varphi(t)=\sum_{k\ge1}\frac{2}{t^{2}+4\pi^{2}k^{2}} .
	\]
	Truncating the geometric series of each summand after $M$ terms,
	\[
		\frac{2}{t^{2}+4\pi^{2}k^{2}}
		=\sum_{n=1}^{M}\frac{2(-1)^{n-1}t^{2n-2}}{(4\pi^{2}k^{2})^{n}}
		+\frac{2(-1)^{M}t^{2M}}{(4\pi^{2}k^{2})^{M}\,(t^{2}+4\pi^{2}k^{2})},
	\]
	and summing over $k$ with Euler's identity
	\[
		\sum_{k\ge1}2(4\pi^{2}k^{2})^{-n}
		=2\zeta(2n)(2\pi)^{-2n}
		=\frac{(-1)^{n-1}B_{2n}}{(2n)!},
	\]
	we get, for every $t>0$,
	\begin{equation}\label{eq:QM}
		Q_M(t):=\varphi(t)-\sum_{n=1}^{M}\frac{B_{2n}}{(2n)!}\,t^{2n-2}
		=2(-1)^{M}\sum_{k\ge1}\frac{t^{2M}}{(4\pi^{2}k^{2})^{M}\,(t^{2}+4\pi^{2}k^{2})} ,
	\end{equation}
	so that $(-1)^{M}Q_M(t)>0$ pointwise.  From \eqref{eq:QM}, $Q_M(t)=O(t^{2M})$ as $t\to0$ and
	$Q_M(t)=O(t^{2M-2})+O(t^{-1})$ as $t\to\infty$; against $|\Delta|\le Ce^{-ct}$ all integrals
	converge absolutely, and
	\[
		L-\sum_{n=1}^{M}c_n=\int_0^\infty Q_M(t)\,\Delta(t)\,dt .
	\]
	The integrand has the fixed sign $(-1)^{M}\cdot(-1)=(-1)^{M+1}$ and is not identically zero,
	which is \eqref{eq:skeleton}.  Applying \eqref{eq:skeleton} at $M$ and $M+1$ shows that $L$
	lies strictly between the $M$-th and $(M+1)$-st partial sums; and $c_{M+1}=(L-\sum_{\le M})
	-(L-\sum_{\le M+1})$ carries the sign $(-1)^{M+1}$.
\end{proof}

\begin{theorem}[Enveloping at an integer mean: Poisson]\label{thm:poi-envelope}
	Let $\lambda\in\Z$, $\lambda\ge1$.  Then, exactly,
	\begin{equation}\label{eq:poi-binet}
		\log\frac{\E|X-\lambda|}{\sqrt{2\lambda/\pi}}
		=-J(\lambda)
		=\int_0^\infty\varphi(t)\,\bigl(-e^{-\lambda t}\bigr)\,dt ,
	\end{equation}
	and the series of Theorem~\ref{thm:poi-expansion} collapses to the odd series
	\[
		\sum_{j\ge1}a_{2j-1}(0)\lambda^{1-2j},
		\qquad
		a_{2j-1}(0)=-\frac{B_{2j}}{(2j-1)2j},
	\]
	whose successive partial sums strictly bracket the left-hand side of \eqref{eq:poi-binet}: for
	every $M\ge0$,
	\[
		(-1)^{M+1}\Bigl[\log\frac{\E|X-\lambda|}{\sqrt{2\lambda/\pi}}
		-\sum_{j=1}^{M}\frac{a_{2j-1}(0)}{\lambda^{2j-1}}\Bigr]>0 .
	\]
\end{theorem}

\begin{proof}
	At $\theta=0$ the even-index coefficients of \eqref{eq:poi-mad} vanish, since
	$a_{2j}(0)=-B_{2j+1}/\bigl(2j(2j+1)\bigr)=0$, and $a_{2j-1}(0)=-B_{2j}/((2j-1)2j)$.

	For \eqref{eq:poi-binet}: by \eqref{eq:poi-closed} with $m=\lambda$,
	$\E|X-\lambda|=2\lambda e^{-\lambda}\lambda^{\lambda}/\Gamma(\lambda+1)$, and by
	\eqref{eq:binet}
	\[
		\log\Gamma(\lambda+1)=\log\lambda+\log\Gamma(\lambda)
		=\Bigl(\lambda+\tfrac12\Bigr)\log\lambda-\lambda+\tfrac12\log2\pi+J(\lambda),
	\]
	so
	\[
	\begin{aligned}
		\log\E|X-\lambda|
		&=\log2+(\lambda+1)\log\lambda-\lambda-\log\Gamma(\lambda+1)\\
		&=\log2+\tfrac12\log\lambda-\tfrac12\log2\pi-J(\lambda),
	\end{aligned}
	\]
	which is \eqref{eq:poi-binet}.  Now apply Proposition~\ref{prop:skeleton} with
	$\Delta(t)=-e^{-\lambda t}$, which is continuous, strictly negative and exponentially
	bounded.  Its power moments are $\int_0^\infty t^{2n-2}e^{-\lambda t}dt
	=(2n-2)!\,\lambda^{1-2n}$, so
	\[
	\begin{aligned}
		c_n
		&=-\frac{B_{2n}}{(2n)!}\,(2n-2)!\,\lambda^{1-2n}\\
		&=-\frac{B_{2n}}{(2n-1)2n}\,\frac1{\lambda^{2n-1}}
		=\frac{a_{2n-1}(0)}{\lambda^{2n-1}},
	\end{aligned}
	\]
	and \eqref{eq:skeleton} is the assertion.
\end{proof}

\section{The negative binomial law}\label{sec:nb}

	Let $X\sim\NB(r,p)$ in the failures parametrisation, with $r>0$ real, $0<p<1$ fixed,
	$q=1-p$,
	\[
		P(k)=\frac{\Gamma(k+r)}{\Gamma(r)\,\Gamma(k+1)}\,p^{r}q^{k},
		\qquad
		\mu=\frac{rq}{p},\qquad \sigma^{2}=\frac{rq}{p^{2}},
	\]
	and let $r\to\infty$; put
	\[
		\nu=\lceil\mu\rceil,\qquad h=h_r=\nu-\mu\in[0,1).
	\]
	By Theorem~\ref{thm:collapse}, $\E|X-\mu|=2(\nu/p)P(\nu)$.

\subsection{The complete expansion}

\begin{theorem}\label{thm:nb-expansion}
	Fix a compact $\mathcal P\subset(0,1)$.  As $r\to\infty$ through the reals, uniformly for
	$p\in\mathcal P$,
	\begin{equation}\label{eq:nb-mad}
	\begin{aligned}
		\E|X-\mu|
		&\sim\frac{1}{p}\sqrt{\frac{2rq}{\pi}}\;
			\exp\Bigl(\sum_{n\ge1}\frac{\tilde a_n(h;p)}{r^{n}}\Bigr),\\
		\tilde a_n
		&=\frac{(-1)^{n+1}}{n(n+1)}
			\Bigl[\bigl(p^{n}-(p/q)^{n}\bigr)B_{n+1}(h)-B_{n+1}\Bigr],
	\end{aligned}
	\end{equation}
	in the truncated-and-uniform sense of \S\ref{sec:prelim}, the estimate being uniform in the
	displacement $h\in[0,1)$.  Note $\tfrac1p\sqrt{2rq/\pi}=\sqrt{2\sigma^{2}/\pi}$.
	Explicitly,
	\begin{equation}\label{eq:nb-first}
	\begin{aligned}
		\tilde a_1&=-\frac1{12}-\frac{p^{2}}{2q}\,B_2(h),\\
		\tilde a_2&=\frac{p^{3}(1+q)}{6q^{2}}\,B_3(h),\\
		\tilde a_3&=\frac1{360}-\frac{p^{4}(1+q+q^{2})}{12q^{3}}\,B_4(h).
	\end{aligned}
	\end{equation}
\end{theorem}

\begin{proof}
	\emph{Step 1: the mass at $\nu$ as a balanced gamma quotient.}  Since
	$\nu+r=\tfrac rp+h$ and $\nu+1=\tfrac{rq}p+h+1$,
	\[
		P(\nu)=\frac{\Gamma\bigl(\tfrac1p\,r+h\bigr)}
		{\Gamma(1\cdot r+0)\;\Gamma\bigl(\tfrac qp\,r+h+1\bigr)}\;p^{r}q^{\nu},
	\]
	a quotient in the format of Lemma~\ref{lem:unequal} with $x=r$ and
	\[
		\lambda_1=\tfrac1p,\ u_1=h;\qquad
		\mu_1=1,\ v_1=0;\qquad
		\mu_2=\tfrac qp,\ v_2=h+1 .
	\]
	The scalings lie in compacts of $(0,\infty)$ for $p\in\mathcal P$, the shifts in $[0,2]$,
	and
	\[
		\Lambda=\frac1p-1-\frac qp=\frac{1-p-q}{p}=0 ,
	\]
	so the lemma applies.

	\emph{Step 2: the elementary constants assemble exactly.}  By the lemma,
	\[
		\Theta=\frac1p\log\frac1p-\frac qp\log\frac qp ,
	\]
	and, with $\nu=\tfrac{rq}p+h$,
	\[
	\begin{aligned}
		\Theta r+r\log p+\nu\log q
		&=r\Bigl[-\frac{\log p}p-\frac qp\log q+\frac qp\log p+\log p+\frac qp\log q\Bigr]
			+h\log q\\
		&=h\log q,
	\end{aligned}
	\]
	since $-\tfrac1p+1+\tfrac qp=0$: the exponential rate cancels identically, as it must ---
	the mass at the mean is only polynomially small.  Next
	\[
		U=\Bigl(h-\tfrac12\Bigr)-\Bigl(0-\tfrac12\Bigr)-\Bigl(h+\tfrac12\Bigr)=-\tfrac12 ,
	\]
	\[
		K=\Bigl(h-\tfrac12\Bigr)\log\frac1p-\Bigl(h+\tfrac12\Bigr)\log\frac qp
		+\frac{1-2}2\log2\pi
		=\log p-\Bigl(h+\tfrac12\Bigr)\log q-\tfrac12\log2\pi ,
	\]
	so that
	\[
		e^{K}\,q^{h}\,r^{-1/2}
		=p\,q^{-h-\frac12}\,q^{h}\,(2\pi r)^{-1/2}
		=\frac{p}{\sqrt{2\pi rq}}
		=\frac1{\sqrt{2\pi\sigma^{2}}} :
	\]
	the Gaussian prefactor, automatically.  Hence
	\begin{equation}\label{eq:nb-mass}
		P(\nu)\sim\frac{1}{\sqrt{2\pi\sigma^{2}}}\,
		\exp\Bigl(\sum_{n\ge1}\frac{A_n(h;p)}{r^{n}}\Bigr),
		\qquad
		A_n=\frac{(-1)^{n+1}}{n(n+1)}\,S_{n+1},
	\end{equation}
	uniformly, with, by \eqref{eq:Sn},
	\[
		S_{n+1}=p^{n}B_{n+1}(h)-B_{n+1}(0)-\Bigl(\frac pq\Bigr)^{n}B_{n+1}(h+1) .
	\]

	\emph{Step 3: splitting off the elementary tail.}  By the difference identity
	\eqref{eq:appell}, $B_{n+1}(h+1)=B_{n+1}(h)+(n+1)h^{n}$, so
	\[
		S_{n+1}=\bigl(p^{n}-(p/q)^{n}\bigr)B_{n+1}(h)-B_{n+1}-(n+1)\Bigl(\frac pq\Bigr)^{n}h^{n},
	\]
	and therefore
	\begin{equation}\label{eq:nb-split}
		A_n=\tilde a_n+\frac{(-1)^{n}}{n}\Bigl(\frac{p}{q}\Bigr)^{n}h^{n}
	\end{equation}
	with $\tilde a_n$ as in \eqref{eq:nb-mad}: a pure Appell part plus an elementary tail.

	\emph{Step 4: the tail is the collapse prefactor.}  By Lemma~\ref{lem:logfactor} with
	$\alpha=q/p$ (compact for $p\in\mathcal P$), for every $M$
	\[
		\sum_{n=1}^{M}\frac{(-1)^{n}}{n}\Bigl(\frac{ph}{qr}\Bigr)^{n}
		=-\log\Bigl(1+\frac{ph}{qr}\Bigr)+O(r^{-M-1})
		=-\log\frac{p\nu}{qr}+O(r^{-M-1}),
	\]
	the last equality because $p\nu=rq+ph$.  On the other hand the collapse prefactor is
	\[
		\frac{2\nu}{p}=\frac{2rq}{p^{2}}\cdot\frac{p\nu}{qr}
		=2\sigma^{2}\cdot\frac{p\nu}{qr} .
	\]
	Truncate the exponent of \eqref{eq:nb-mass} after $M$ terms (cost: relative
	$O(r^{-M-1})$, by the lemma), split it by \eqref{eq:nb-split}, and replace the truncated
	tail by the exact factor $qr/(p\nu)$ (cost: another $O(r^{-M-1})$; the factor is bounded
	above and below by Lemma~\ref{lem:logfactor}).  Multiplying by the prefactor $2\nu/p$, the
	two factors cancel exactly, leaving
	\[
		\E|X-\mu|=\frac{2\nu}{p}\,P(\nu)
		=2\sigma^{2}\cdot\frac1{\sqrt{2\pi\sigma^{2}}}\,
		\exp\Bigl(\sum_{n=1}^{M}\frac{\tilde a_n}{r^{n}}\Bigr)\bigl(1+O(r^{-M-1})\bigr),
	\]
	and $2\sigma^{2}(2\pi\sigma^{2})^{-1/2}=\sqrt{2\sigma^{2}/\pi}=\tfrac1p\sqrt{2rq/\pi}$,
	which is \eqref{eq:nb-mad}.

	For \eqref{eq:nb-first}, reduce the weights by $q-1=-p$:
	$p-(p/q)=p(q-1)/q=-p^{2}/q$;
	$p^{2}-(p/q)^{2}=p^{2}(q^{2}-1)/q^{2}=-p^{3}(1+q)/q^{2}$;
	$p^{3}-(p/q)^{3}=-p^{4}(1+q+q^{2})/q^{3}$; and $B_2=\tfrac16$, $B_4=-\tfrac1{30}$.
\end{proof}

\begin{remark}\label{rem:nb-pattern}
	The constant term $-B_{n+1}$ enters \eqref{eq:nb-mad} with the sign opposite to the binomial
	case \cite[Thm 4.2]{elezovic_mad}, because $\Gamma(r)$ sits in the \emph{denominator} of the
	mass where the binomial's $\Gamma(N+1)$ sat in the numerator.  The weight pattern across the
	family is
	\[
		\underbrace{p^{-n}+(-1)^{n+1}q^{-n}}_{\text{binomial}},\qquad
		\underbrace{1}_{\text{Poisson}},\qquad
		\underbrace{p^{n}-(p/q)^{n}}_{\text{negative binomial}} ,
	\]
	and the weights carry $q^{-n}$, so the expansion degrades as $q\to0$: the effective
	parameter is $\asymp1/(rq)\asymp1/(p\mu)$.  This is the limiting regime in which
	\eqref{eq:nb-mad} must match the Poisson expansion \eqref{eq:poi-mad} as $rq\to\lambda$; a
	uniform matching is left open.
\end{remark}

\subsection{The enveloping theorem}

\begin{theorem}[Enveloping at an integer mean: negative binomial]\label{thm:nb-envelope}
	Let $\mu=rq/p\in\Z$, $\mu\ge1$ (so $h=0$, $\nu=\mu$).  Then, exactly,
	\begin{equation}\label{eq:nb-binet}
		\log\frac{\E|X-\mu|}{\frac1p\sqrt{2rq/\pi}}
		=\int_0^\infty\varphi(t)\,\Delta_r(t)\,dt,
		\qquad
		\Delta_r(t)=e^{-\frac rp t}-e^{-rt}-e^{-\frac{rq}p t} ,
	\end{equation}
	and $\Delta_r(t)<0$ for every $t>0$.  The expansion \eqref{eq:nb-mad} collapses at $h=0$ to
	the odd series with coefficients
	\[
		\tilde a_{2j-1}(0)=\frac{B_{2j}}{(2j-1)2j}
		\Bigl[p^{2j-1}-\Bigl(\frac pq\Bigr)^{2j-1}-1\Bigr],
		\qquad (-1)^{j}\,\tilde a_{2j-1}(0)>0 ,
	\]
	and its successive partial sums strictly bracket the left-hand side of \eqref{eq:nb-binet}.
\end{theorem}

\begin{proof}
	\emph{Parity and sign.}  At $h=0$ the even coefficients of \eqref{eq:nb-mad} vanish
	($B_{2j+1}(0)=B_{2j+1}=0$), and
	$\tilde a_{2j-1}(0)=\bigl[(p^{2j-1}-(p/q)^{2j-1})B_{2j}-B_{2j}\bigr]/((2j-1)2j)$, which is
	the displayed formula.  The bracket is strictly negative: $p^{2j-1}<1$ and
	$(p/q)^{2j-1}>0$, so $p^{2j-1}-(p/q)^{2j-1}-1<0$.  Since
	$\operatorname{sign}B_{2j}=(-1)^{j+1}$, the product has sign $(-1)^{j}$.

	\emph{The exact identity.}  With $\mu\in\Z$, $\nu=\mu$ and $\mu+r=r/p$,
	\[
		\E|X-\mu|=\frac{2\mu}{p}\,
		\frac{\Gamma(r/p)}{\Gamma(r)\,\Gamma(\mu+1)}\,p^{r}q^{\mu}.
	\]
	Apply Binet's formula \eqref{eq:binet} to the three gamma factors, using
	$\log\Gamma(\mu+1)=(\mu+\tfrac12)\log\mu-\mu+\tfrac12\log2\pi+J(\mu)$ as in
	Theorem~\ref{thm:poi-envelope}.  The Binet remainders contribute
	$J(r/p)-J(r)-J(\mu)$, which is the right-hand side of \eqref{eq:nb-binet} with
	$\mu=rq/p$.  The elementary parts assemble to the leading factor: the linear terms give
	\[
		-\frac rp+r+\mu=r\Bigl(-\frac1p+1+\frac qp\Bigr)=0,
	\]
	the $\log$-linear terms give
	\[
		\frac rp\log\frac rp-r\log r-\mu\log\mu+r\log p+\mu\log q=0
	\]
	(collect the coefficients of $\log r$, $\log p$, $\log q$ separately: each is a multiple of
	$-\tfrac1p+1+\tfrac qp=0$), and the residual half-logs and constants give
	\[
	\begin{aligned}
		-\tfrac12\Bigl[\log\frac rp-\log r+\log\mu\Bigr]
		&+\log\frac{2\mu}{p}-\tfrac12\log2\pi\\
		&=\log2+\tfrac12\log(rq)-\log p-\tfrac12\log2\pi\\
		&=\log\Bigl(\frac1p\sqrt{\frac{2rq}{\pi}}\Bigr),
	\end{aligned}
	\]
	using $\log(r/p)-\log r+\log\mu=\log(rq)-2\log p$.  This proves \eqref{eq:nb-binet}.

	\emph{Negativity.}  Since $p<1$, $\tfrac rp>\tfrac{rq}p$, so
	$e^{-rt/p}<e^{-rqt/p}$ for $t>0$ and $\Delta_r(t)<-e^{-rt}<0$.

	\emph{Enveloping.}  Apply Proposition~\ref{prop:skeleton} with $\Delta=\Delta_r$.  The
	moments are
	\[
	\begin{aligned}
		c_n
		&=\frac{B_{2n}}{(2n)!}\,(2n-2)!\,
			\Bigl[\Bigl(\frac rp\Bigr)^{1-2n}-r^{1-2n}
				-\Bigl(\frac{rq}p\Bigr)^{1-2n}\Bigr]\\
		&=\frac{B_{2n}}{(2n-1)2n}\cdot
			\frac{p^{2n-1}-1-(p/q)^{2n-1}}{r^{2n-1}}
		=\frac{\tilde a_{2n-1}(0)}{r^{2n-1}},
	\end{aligned}
	\]
	and \eqref{eq:skeleton} is the assertion.
\end{proof}

\begin{remark}
	The bracket $\Delta_r$ has the same three-exponential shape as the binomial's
	$e^{-Nt}-e^{-Npt}-e^{-Nqt}$ \cite[\S5]{elezovic_mad}, with the balanced partition
	$\tfrac1p=1+\tfrac qp$ in place of $1=p+q$: the largest gamma argument against the two
	smaller ones, negativity being immediate because the largest argument produces the smallest
	exponential.
\end{remark}

\section{The hypergeometric law}\label{sec:hyp}

	Let $X\sim\Hyp(N,K,n)$,
	\[
		P(k)=\frac{\binom Kk\binom{N-K}{n-k}}{\binom Nn},
		\qquad
		\mu=\frac{nK}{N},\qquad
		\sigma^{2}=\frac{nK(N-K)(N-n)}{N^{2}(N-1)} ,
	\]
	and let $N\to\infty$ with the \emph{margins} in fixed proportions: write
	\[
		\kappa:=\frac KN,\qquad \eta:=\frac nN ,
	\]
	and assume $(\kappa,\eta)$ ranges over a compact subset of $(0,1)^{2}$.  Introduce the four
	\emph{cell proportions} of the $2\times2$ table with margins $\kappa,\eta$ at independence,
	\begin{equation}\label{eq:cells}
		a=\kappa\eta,\qquad b=\kappa(1-\eta),\qquad c=\eta(1-\kappa),\qquad d=(1-\kappa)(1-\eta),
	\end{equation}
	so that
	\begin{equation}\label{eq:independence}
		a+b+c+d=1,\qquad ad=bc,\qquad a+b=\kappa,\ a+c=\eta,\ c+d=1-\kappa,\ b+d=1-\eta .
	\end{equation}
	Put
	\[
		\nu=\lceil\mu\rceil,\qquad h=\nu-\mu\in[0,1),\qquad L:=N-K-n .
	\]
	The relevance of \eqref{eq:cells} is that at $k=\mu$ the four entries
	$k,\ K-k,\ n-k,\ L+k$ of the table underlying $P(k)$ are \emph{exactly} $Na,Nb,Nc,Nd$: the
	independence table.  The identity $ad=bc$ is the independence identity, and it is what makes
	every elementary constant below assemble.

	By Theorem~\ref{thm:collapse}, $\E|X-\mu|=2g(\nu)P(\nu)$ with $g(\nu)=\nu(\nu+L)/N$.

\subsection{The complete expansion}

\begin{theorem}\label{thm:hyp-expansion}
	As $N\to\infty$ with $(\kappa,\eta)=(K/N,n/N)$ in a compact subset of $(0,1)^{2}$,
	\begin{equation}\label{eq:hyp-mad}
	\begin{aligned}
		\E|X-\mu|
		&\sim\sqrt{\frac{2Nad}{\pi}}\;
			\exp\Bigl(\sum_{m\ge1}\frac{\tilde a_m(h)}{N^{m}}\Bigr),\\
		\tilde a_m
		&=\frac{(-1)^{m+1}}{m(m+1)}
			\Bigl[(W_m-1)\,B_{m+1}-\widehat\Xi_m\,B_{m+1}(h)\Bigr],
	\end{aligned}
	\end{equation}
	uniformly, in the truncated sense of \S\ref{sec:prelim}, where
	\[
	\begin{aligned}
		W_m&=\kappa^{-m}+(1-\kappa)^{-m}+\eta^{-m}+(1-\eta)^{-m},\\
		\widehat\Xi_m&=\bigl(a^{-m}+d^{-m}\bigr)+(-1)^{m+1}\bigl(b^{-m}+c^{-m}\bigr).
	\end{aligned}
	\]
	Here $Nad=nK(N-K)(N-n)/N^{3}=\sigma^{2}\,(N-1)/N$.  In particular, using $ad=bc$,
	\begin{equation}\label{eq:hyp-first}
		\tilde a_1=\frac1{12}\Bigl[\frac1{\kappa(1-\kappa)}+\frac1{\eta(1-\eta)}-1\Bigr]
		-\frac{B_2(h)}{2\,ad} .
	\end{equation}
\end{theorem}

\begin{proof}
	\emph{Step 1: nine gammas, balanced.}  Writing the three binomial coefficients through
	factorials,
	\[
		P(\nu)=\frac{\Gamma(K+1)\,\Gamma(N-K+1)\,\Gamma(n+1)\,\Gamma(N-n+1)}
		{\Gamma(\nu+1)\,\Gamma(K-\nu+1)\,\Gamma(n-\nu+1)\,\Gamma(L+\nu+1)\,\Gamma(N+1)} .
	\]
	Since $\nu=Na+h$, the nine arguments are, in the format of Lemma~\ref{lem:unequal} with
	$x=N$,
	\[
		\begin{array}{llll}
			\text{numerator:} &
			(\kappa,\,1),\quad (1-\kappa,\,1),\quad (\eta,\,1),\quad (1-\eta,\,1);\\[2pt]
			\text{denominator:} &
			(a,\,h+1),\quad (b,\,1-h),\quad (c,\,1-h),\quad (d,\,h+1),\quad (1,\,1),
		\end{array}
	\]
	each pair being (scaling, shift): for instance
	$K-\nu+1=N\kappa-Na-h+1=Nb+(1-h)$ and $L+\nu+1=Nd+(h+1)$.  The scalings lie in compacts of
	$(0,\infty)$, the shifts in $[0,2]$, and the balance holds:
	$\kappa+(1-\kappa)+\eta+(1-\eta)=2=a+b+c+d+1$.

	\emph{Step 2: the elementary constants assemble through $ad=bc$.}  First $\Theta=0$:
	\[
	\begin{aligned}
		\sum_{\text{cells}}(\text{cell})\log(\text{cell})
		={}&a(\log\kappa+\log\eta)+b(\log\kappa+\log(1-\eta))\\
		&+c(\log\eta+\log(1-\kappa))+d(\log(1-\kappa)+\log(1-\eta)),
	\end{aligned}
	\]
	and collecting the four logarithms by \eqref{eq:independence} this equals
	$\kappa\log\kappa+(1-\kappa)\log(1-\kappa)+\eta\log\eta+(1-\eta)\log(1-\eta)
	=\sum_{\text{margins}}(\cdot)\log(\cdot)$; the total's contribution $1\cdot\log1$ vanishes.
	So the exponential rate is identically zero: the entropy of the independence table equals
	the entropy of its margins.  Next,
	\[
		U=4\cdot\tfrac12-\Bigl[\Bigl(h+\tfrac12\Bigr)+\Bigl(\tfrac12-h\Bigr)
		+\Bigl(\tfrac12-h\Bigr)+\Bigl(h+\tfrac12\Bigr)+\tfrac12\Bigr]=2-\tfrac52=-\tfrac12 ,
	\]
	and, using first $\kappa(1-\kappa)\eta(1-\eta)=(\kappa\eta)\bigl((1-\kappa)(1-\eta)\bigr)=ad$
	and then $ad=bc$,
	\[
	\begin{aligned}
		K_{\mathrm{el}}
		&=\tfrac12\log(ad)
			-\Bigl(h+\tfrac12\Bigr)\log(ad)-\Bigl(\tfrac12-h\Bigr)\log(bc)
			+\frac{4-5}2\log2\pi\\
		&=-\tfrac12\log(ad)-\tfrac12\log2\pi .
	\end{aligned}
	\]
	Hence $e^{K_{\mathrm{el}}}\,N^{-1/2}=(2\pi Nad)^{-1/2}$, and
	\begin{equation}\label{eq:hyp-mass}
		P(\nu)\sim\frac{1}{\sqrt{2\pi Nad}}\,
		\exp\Bigl(\sum_{m\ge1}\frac{A_m(h)}{N^{m}}\Bigr),
		\qquad A_m=\frac{(-1)^{m+1}}{m(m+1)}S_{m+1},
	\end{equation}
	uniformly, with
	\[
	\begin{aligned}
		S_{m+1}
		&=W_m\,B_{m+1}\\
		&\quad-(a^{-m}+d^{-m})B_{m+1}(h+1)\\
		&\quad-(b^{-m}+c^{-m})B_{m+1}(1-h)-B_{m+1}(1).
	\end{aligned}
	\]

	\emph{Step 3: the elementary tail has two factors.}  By the two identities
	\eqref{eq:appell},
	\[
	\begin{aligned}
		S_{m+1}
		&=(W_m-1)B_{m+1}-\widehat\Xi_m\,B_{m+1}(h)\\
		&\quad -(m+1)\,h^{m}\bigl(a^{-m}+d^{-m}\bigr),
	\end{aligned}
	\]
	so that
	\begin{equation}\label{eq:hyp-split}
		A_m=\tilde a_m+\frac{(-1)^{m}}{m}\,h^{m}\bigl(a^{-m}+d^{-m}\bigr):
	\end{equation}
	a pure Appell part plus \emph{two} elementary tails, one for each diagonal cell --- the two
	entries of the table that grow with $k$.

	\emph{Step 4: the quadratic prefactor cancels both.}  By Lemma~\ref{lem:logfactor}, applied
	twice (with $\alpha=a$ and $\alpha=d$),
	\[
	\begin{aligned}
		\sum_{m=1}^{M_0}\frac{(-1)^{m}h^{m}}{m}
			\Bigl(\frac1{(Na)^{m}}+\frac1{(Nd)^{m}}\Bigr)
		&=-\log\Bigl(1+\frac h{Na}\Bigr)-\log\Bigl(1+\frac h{Nd}\Bigr)\\
		&\quad+O(N^{-M_0-1})\\
		&=-\log\frac{\nu\,(\nu+L)}{N^{2}ad}+O(N^{-M_0-1}),
	\end{aligned}
	\]
	since $Na+h=\nu$ and $Nd+h=\nu+L$.  The collapse prefactor is exactly
	\[
		2g(\nu)=\frac{2\nu(\nu+L)}{N}
		=2Nad\cdot\frac{\nu(\nu+L)}{N^{2}ad}.
	\]
	the quadratic $g$ is the product of the two diagonal factors.  Truncating the exponent of
	\eqref{eq:hyp-mass}, splitting by \eqref{eq:hyp-split}, replacing the truncated tails by the
	exact factor $N^{2}ad/\bigl(\nu(\nu+L)\bigr)$, and multiplying by $2g(\nu)$, the factors
	cancel exactly, leaving
	\[
		\E|X-\mu|=2Nad\cdot\frac1{\sqrt{2\pi Nad}}\,
		\exp\Bigl(\sum_{m=1}^{M_0}\frac{\tilde a_m}{N^{m}}\Bigr)\bigl(1+O(N^{-M_0-1})\bigr),
	\]
	and $2Nad(2\pi Nad)^{-1/2}=\sqrt{2Nad/\pi}$, which is \eqref{eq:hyp-mad}.

	For \eqref{eq:hyp-first}: $W_1=\kappa^{-1}+(1-\kappa)^{-1}+\eta^{-1}+(1-\eta)^{-1}
	=\frac1{\kappa(1-\kappa)}+\frac1{\eta(1-\eta)}$, and, by $ad=bc$,
	\[
		\widehat\Xi_1=\frac1a+\frac1b+\frac1c+\frac1d
		=\frac{a+d}{ad}+\frac{b+c}{bc}=\frac{a+b+c+d}{ad}=\frac1{ad},
	\]
	so $\tilde a_1=\tfrac12\bigl[(W_1-1)B_2-\widehat\Xi_1B_2(h)\bigr]
	=\tfrac1{12}(W_1-1)-B_2(h)/(2ad)$.
\end{proof}

\begin{remark}\label{rem:hyp-pattern}
	The coefficients speak the binomial grammar over a $2\times2$ table: the diagonal cells,
	which grow with $k$, enter $\widehat\Xi_m$ with $+$; the off-diagonal cells, which shrink,
	enter with the parity sign $(-1)^{m+1}$ --- exactly as $p^{-m}$ and $q^{-m}$ do in the
	binomial weight $p^{-m}+(-1)^{m+1}q^{-m}$ \cite{elezovic_mad}.  The margins and the total
	play the role that the numerator $\Gamma(N+1)$ played there, supplying the constant
	$(W_m-1)B_{m+1}$.  In $\tilde a_1$ the binomial's $\tfrac1{12}-B_2(h)/(2pq)$ reappears with
	$pq\mapsto ad$ and the margins entering the constant term.

	The natural scale of the expansion is $Nad$, not $\sigma^{2}$: the finite-population factor
	$N/(N-1)$ stays inside $\sigma^{2}=Nad\cdot N/(N-1)$ and never enters the coefficients.
\end{remark}

\subsection{The enveloping theorem}

\begin{theorem}[Enveloping at an integer mean: hypergeometric]\label{thm:hyp-envelope}
	Let $\mu=nK/N\in\Z$, $1\le\mu$, with $(\kappa,\eta)\in(0,1)^{2}$ (so $h=0$, $\nu=\mu=Na$).
	Then, exactly,
	\begin{equation}\label{eq:hyp-binet}
		\log\frac{\E|X-\mu|}{\sqrt{2Nad/\pi}}
		=\int_0^\infty\varphi(t)\,\Delta_N(t)\,dt,
	\end{equation}
	with
	\[
	\begin{aligned}
		\Delta_N(t)=
		& e^{-Kt}+e^{-(N-K)t}+e^{-nt}+e^{-(N-n)t}\\
		&-\Bigl[e^{-Nat}+e^{-Nbt}+e^{-Nct}+e^{-Ndt}+e^{-Nt}\Bigr],
	\end{aligned}
	\]
	and $\Delta_N(t)<0$ for every $t>0$.  The expansion \eqref{eq:hyp-mad} collapses at $h=0$
	to the odd series with coefficients
	\[
	\begin{aligned}
		\tilde a_{2j-1}(0)
		&=\frac{B_{2j}}{(2j-1)2j}
			\Bigl[W_{2j-1}-1-\bigl(a^{1-2j}+b^{1-2j}+c^{1-2j}+d^{1-2j}\bigr)\Bigr],\\
		(-1)^{j}\,\tilde a_{2j-1}(0)&>0,
	\end{aligned}
	\]
	and its successive partial sums strictly bracket the left-hand side of
	\eqref{eq:hyp-binet}.
\end{theorem}

\begin{proof}
	\emph{Parity and sign.}  At $h=0$ the even coefficients vanish, since for even $m$ both
	$B_{m+1}$ and $B_{m+1}(0)$ are odd-index Bernoulli numbers; for odd $m=2j-1$ the weight
	$\widehat\Xi_{2j-1}$ has all four signs $+$, giving the displayed formula.  Each cell is
	majorised by a margin containing it:
	\[
		a<\kappa,\qquad b<1-\eta,\qquad c<\eta,\qquad d<1-\kappa,
	\]
	so $a^{1-2j}>\kappa^{1-2j}$ etc., whence
	$W_{2j-1}-\sum_{\text{cells}}(\cdot)^{1-2j}<0$ and the bracket is $<-1<0$; with
	$\operatorname{sign}B_{2j}=(-1)^{j+1}$ the product has sign $(-1)^{j}$.

	\emph{The exact identity.}  Write each of the nine factors of
	$\E|X-\mu|=2g(\mu)P(\mu)$ through Binet's formula \eqref{eq:binet}, using
	$\log\Gamma(y+1)=\log y+\log\Gamma(y)$ at the nine integer arguments
	$y=K,\,N-K,\,n,\,N-n;\ Na,\,Nb,\,Nc,\,Nd,\,N$.  The Binet remainders contribute exactly
	$\sum_{\text{margins}}J(\cdot)-\sum_{\text{cells}}J(\cdot)-J(N)$, which is the
	right-hand side of \eqref{eq:hyp-binet}.  The elementary parts assemble to the leading
	factor, exactly as in Step~2 of Theorem~\ref{thm:hyp-expansion}.  Writing
	$\log\Gamma(y+1)=(y+\tfrac12)\log y-y+\tfrac12\log2\pi+J(y)$ for each of the nine factors and
	taking margins minus cells-and-total: the terms $-y$ cancel, both groups summing to $2N$; the
	$y\log y$ terms cancel by the entropy identity of Step~2; the constants leave
	$-\tfrac12\log2\pi$ (four minus five copies of $\tfrac12\log2\pi$); and the half-logs
	$\tfrac12\log y$ sum to
	\[
		\tfrac12\Bigl[\log\bigl(K(N-K)n(N-n)\bigr)-\log\bigl(N^5abcd\bigr)\Bigr]
		=\tfrac12\log\frac{N^{4}ad}{N^{5}(ad)^{2}}=-\tfrac12\log(Nad),
	\]
	using $K(N-K)n(N-n)=N^{4}ad$ and $abcd=(ad)^{2}$ (both from $ad=bc$).  With the prefactor
	$\log\bigl(2g(\mu)\bigr)=\log(2Nad)$ --- note $g(\mu)=Na\cdot Nd/N=Nad$ at $h=0$ --- the
	elementary total is
	$\log(2Nad)-\tfrac12\log(Nad)-\tfrac12\log2\pi=\log\sqrt{2Nad/\pi}$.

	\emph{Negativity.}  Pair each cell against a margin containing it, as above: from $a<\kappa$
	we get $e^{-Nat}>e^{-N\kappa t}=e^{-Kt}$, and likewise
	$e^{-Nbt}>e^{-(N-n)t}$, $e^{-Nct}>e^{-nt}$, $e^{-Ndt}>e^{-(N-K)t}$.  Summing,
	$\Delta_N(t)<-e^{-Nt}<0$.

	\emph{Enveloping.}  Apply Proposition~\ref{prop:skeleton} with $\Delta=\Delta_N$: the
	moments give
	\[
		c_j=\frac{B_{2j}}{(2j-1)2j}\cdot\frac1{N^{2j-1}}
		\Bigl[W_{2j-1}-1-\sum_{\text{cells}}(\cdot)^{1-2j}\Bigr]
		=\frac{\tilde a_{2j-1}(0)}{N^{2j-1}},
	\]
	and \eqref{eq:skeleton} is the assertion.
\end{proof}

\section{The family pattern}\label{sec:family}

	The four cases are instances of one scheme, which we record as a table.  Throughout,
	$\nu=\lceil\mu\rceil$ and $h$ is the ceiling defect of the mean ($\theta=\{\lambda\}=1-h$
	for the Poisson when $\lambda\notin\Z$; the coefficients are written in each law's natural
	variable).

\begin{center}\scriptsize
	\begin{tabular}{lllll}
	\hline
	law & ratio $\deg$ & $g(k)$ & leading factor & weight of $B_{m+1}(\cdot)$\\
	\hline
	$\Bin(N,p)$ & $1$ & $kq$ & $\sqrt{2Npq/\pi}$ & $p^{-m}+(-1)^{m+1}q^{-m}$\\[2pt]
	$\Poi(\lambda)$ & $1$ & $k$ & $\sqrt{2\lambda/\pi}$ & $1$\\[2pt]
	$\NB(r,p)$ & $1$ & $k/p$ & $\tfrac1p\sqrt{2rq/\pi}$ & $(p/q)^{m}-p^{m}$\\[2pt]
	$\Hyp(N,K,n)$ & $2$ & $k(k+L)/N$ & $\sqrt{2Nad/\pi}$ &
		$\begin{array}[t]{@{}l@{}}
			(a^{-m}+d^{-m})\\
			{}+(-1)^{m+1}(b^{-m}+c^{-m})
		\end{array}$\\
	\hline
	\end{tabular}
	\end{center}

	The last column records the weight $w_m$ of $B_{m+1}(h)$ in the common normalisation
	$S_{m+1}=(\mathrm{const})\,B_{m+1}-w_m\,B_{m+1}(h)-(\text{elementary tail})$, the global factor
	$\tfrac{(-1)^{m+1}}{m(m+1)}$ being kept outside; entries that grow with $k$ enter $w_m$ with
	sign $+$, entries that shrink with the parity sign $(-1)^{m+1}$.  For the negative binomial
	this is $w_m=(p/q)^m-p^m$, i.e.\ minus the coefficient $p^m-(p/q)^m$ of $B_{m+1}(h)$ displayed
	in Theorem~\ref{thm:nb-expansion}, where the same bracket is written in the opposite order.

\begin{remark}[The scheme]\label{rem:scheme}
	(i) \emph{Collapse}: one telescoping \eqref{eq:tele}, with $\deg g$ equal to the degree of
	the Katz--Ord ratio (Remark~\ref{rem:ord-degree}).
	(ii) \emph{Cancellation}: the elementary tail of the Stirling expansion of $P(\nu)$ is
	$\sum(-1)^m h^m(\text{growing cells})^{-m}/m$, one term per entry of the law that grows with
	$k$, and it is cancelled exactly by the collapse prefactor $2g(\nu)$, which is the product of
	the size-bias factors of precisely those entries --- one for the binomial ($\nu$), none for
	the Poisson, one for the negative binomial ($\nu$), two for the hypergeometric
	($\nu$ and $\nu+L$).
	(iii) \emph{Pure Appell}: what remains is a constant-weight $B_{m+1}$ plus a cell-weight
	$B_{m+1}(h)$, growing cells entering with $+$ and shrinking cells with the parity sign.
	(iv) \emph{Integer mean}: parity collapse, strict sign alternation, and an enveloping odd
	series, always by Proposition~\ref{prop:skeleton}, with a bracket
	$\Delta=\sum(\text{numerator exponentials})-\sum(\text{denominator exponentials})$ that is
	strictly negative: for the Poisson it is the single term $-e^{-\lambda t}$; for the negative
	binomial the one numerator argument is the largest; for the hypergeometric each cell exponential
	is paired below a margin exponential, the total $N$ contributing a further $-e^{-Nt}$.
\end{remark}

\section{Concluding remarks}\label{sec:conclusion}

	The closed forms \eqref{eq:collapse} are classical case by case.  For the binomial the
	history runs from De~Moivre onward \cite{diaconis_zabell}.  For the Poisson the closed form
	\eqref{eq:poi-closed} was obtained independently by Ramasubban~\cite{ramasubban1958} and
	Crow~\cite{crow1958}.  The hypergeometric case of \eqref{eq:collapse}, quadratic $g$
	included, is Ramasubban's formula (3.3) \cite[p.~554]{ramasubban1958}, and is recorded,
	together with the negative binomial case, by Kamat~\cite{kamat1965} (see
	\cite[Ch.~5--6]{jkk}).  General absolute-moment representations for discrete laws are in
	Katti~\cite{katti1960}, and Kamat~\cite{kamat1965} also covers the logarithmic distribution,
	which lies outside the framework used here.

	On the asymptotic side, Johnson's expansion \cite{johnson1957} of the ratio of the mean deviation to the
	standard deviation for the binomial was generalised to a class of laws by
	Kamat~\cite{kamat1966}; both give first-order information only, whereas the expansions here are to all orders.  Diaconis and Zabell
	\cite[\S6--7]{diaconis_zabell} derive the De~Moivre-type identities for the Poisson and the
	binomial through Charlier and Krawtchouk polynomials, and explicitly leave the negative
	binomial (Meixner) and hypergeometric (Hahn) cases as a programme.  The size-bias survey of
	Arratia, Goldstein and Kochman~\cite{agk} does not touch the mean absolute deviation.  The
	recent work of Ruzankin~\cite{ruzankin2020} gives exact recurrences for the absolute central
	moments of the Poisson law, reproducing the closed form \eqref{eq:poi-closed}, but no
	asymptotic expansion.  For the hypergeometric law, Ai and Pelekis~\cite{ai_pelekis2026}
	in a study of hypergeometric tail bounds obtained a one-sided lower bound on the mean absolute
	deviation as an auxiliary result, and
	Ouimet~\cite{ouimet2023} treats the local expansion and median of the negative binomial rather
	than its mean deviation.

	The telescoping \eqref{eq:tele} and the Ord degree of $g$ belong to the discrete
	Pearson\,/\,Ord system (Cacoullos--Papathanasiou~\cite{cacoullos_pap1989},
	Korwar~\cite{korwar1991}, Afendras--Papa\-da\-tos--Papathanasiou~\cite{afendras_pp2011}).  The new
	contributions here are the reading of that telescoping through size biasing, giving the
	universal form \eqref{eq:universal} with its parameter shifts (the lattice-exact form of the
	von Bortkiewicz phenomenon, Remark~\ref{rem:bortkiewicz}); the complete expansions
	\eqref{eq:poi-mad}, \eqref{eq:nb-mad}, \eqref{eq:hyp-mad} with closed coefficients and exact
	lattice oscillation; the identification of the cancellation mechanism with the size-bias
	prefactors across the family; and the enveloping theorems.

	Three degenerations connect the four expansions: $p=\lambda/N$ takes the binomial to the
	Poisson, $rq\to\lambda$ the negative binomial to the Poisson, and $\eta\to0$ with $n$ fixed
	the hypergeometric to the binomial.  In each case the effective parameter of the expansion
	degrades exactly along the limiting regime, and a matched, uniform description would require
	the corresponding uniform expansions.  We leave the matching open.

	For the binomial, the companion paper on prescribed centres treats the fold in the tail,
	where the collapse fails and the correct object is a weighted stop-loss tail with Eulerian
	structure.  The corresponding programme for the present family is open, in particular for
	the hypergeometric law, where the prescribed-centre problem is the non-central $2\times2$
	table.

\section*{Declarations}

	\textbf{Funding.}\enspace The author did not receive support from any organization for the
	submitted work.

	\textbf{Competing interests.}\enspace The author declares no competing interests.

	\textbf{Data availability.}\enspace No datasets were generated or analysed during the current
	study.


\end{document}